\documentclass[11pt]{amsart}
\usepackage{amscd}
\usepackage{amsmath}
\usepackage{amsxtra}
\usepackage{amsfonts}
\usepackage{amssymb}

\newtheorem{theorem}{Theorem}[section]

\newtheorem{lemma}[theorem]{Lemma}
\newtheorem{proposition}[theorem]{Proposition}

\theoremstyle{definition}
\newtheorem{definition}[theorem]{Definition}

\theoremstyle{remark}

\renewcommand{\theclaim}{\textup{\theclaim}}

\newtheorem*{acknowledgements}{Acknowledgements}

\numberwithin{equation}{section}

\def\openone

{\mathchoice

{\hbox{\upshape \small1\kern-3.3pt\normalsize1}}

{\hbox{\upshape \small1\kern-3.3pt\normalsize1}}

{\hbox{\upshape \tiny1\kern-2.3pt\SMALL1}}

{\hbox{\upshape \Tiny1\kern-2pt\tiny1}}}

\makeatletter

\newbox\ipbox

\newcommand{\ip}[2]{\left\langle #1\, , \,#2\right\rangle}
\newcommand{\diracb}[1]{\left\langle #1\mathrel{\mathchoice

{\setbox\ipbox=\hbox{$\displaystyle \left\langle\mathstrut
#1\right.$}

\vrule height\ht\ipbox width0.25pt depth\dp\ipbox}

{\setbox\ipbox=\hbox{$\textstyle \left\langle\mathstrut
#1\right.$}

\vrule height\ht\ipbox width0.25pt depth\dp\ipbox}

{\setbox\ipbox=\hbox{$\scriptstyle \left\langle\mathstrut
#1\right.$}

\vrule height\ht\ipbox width0.25pt depth\dp\ipbox}

{\setbox\ipbox=\hbox{$\scriptscriptstyle \left\langle\mathstrut
#1\right.$}

\vrule height\ht\ipbox width0.25pt depth\dp\ipbox}

}\right. }

\newcommand{\dirack}[1]{\left. \mathrel{\mathchoice

{\setbox\ipbox=\hbox{$\displaystyle \left.\mathstrut
#1\right\rangle$}

\vrule height\ht\ipbox width0.25pt depth\dp\ipbox}

{\setbox\ipbox=\hbox{$\textstyle \left.\mathstrut
#1\right\rangle$}

\vrule height\ht\ipbox width0.25pt depth\dp\ipbox}

{\setbox\ipbox=\hbox{$\scriptstyle \left.\mathstrut
#1\right\rangle$}

\vrule height\ht\ipbox width0.25pt depth\dp\ipbox}

{\setbox\ipbox=\hbox{$\scriptscriptstyle \left.\mathstrut
#1\right\rangle$}

\vrule height\ht\ipbox width0.25pt depth\dp\ipbox}

} #1\right\rangle}

\newcommand{\cj}[1]{\overline{#1}}

\newcommand{\bz}{\mathbb{Z}}

\newcommand{\br}{\mathbb{R}}

\newcommand{\bn}{\mathbb{N}}

\newcommand{\beq}{\begin{equation}}
\newcommand{\eeq}{\end{equation}}

\def\blfootnote{\xdef\@thefnmark{}\@footnotetext}

\renewcommand{\mod}{\operatorname{mod}}

\input xy
\xyoption{all}
\usepackage{amssymb}

\def\-{^{-1}}

\begin{document}

\title[Hadamard triples generate self-affine spectral measures]{Hadamard triples generate self-affine spectral measures}
\author{Dorin Ervin Dutkay}

\address{[Dorin Ervin Dutkay] University of Central Florida\\
	Department of Mathematics\\
	4000 Central Florida Blvd.\\
	P.O. Box 161364\\
	Orlando, FL 32816-1364\\
U.S.A.\\} \email{Dorin.Dutkay@ucf.edu}

\author{John Haussermann}

\address{[John Haussermann] University of Central Florida\\
	Department of Mathematics\\
	4000 Central Florida Blvd.\\
	P.O. Box 161364\\
	Orlando, FL 32816-1364\\
U.S.A.\\} \email{jhaussermann@knights.ucf.edu}

\author{Chun-Kit Lai}

\address{[Chun-Kit Lai] Department of Mathematics, San Francisco State University,
1600 Holloway Avenue, San Francisco, CA 94132.}

 \email{cklai@sfsu.edu}

\thanks{}
\subjclass[2010]{Primary 42B05, 42A85, 28A25.}
\keywords{Hadamard triples, quasi-product form, self-affine sets, spectral measure}

\begin{abstract}
Let $R$ be an expanding matrix with integer entries and let $B,L$ be finite integer digit sets so that $(R,B,L)$ form a Hadamard triple on ${\br}^d$. We prove that the associated self-affine measure $\mu = \mu(R,B)$ is a spectral measure, which means it admits an orthonormal bases of exponential functions in $L^2(\mu)$. This settles a long-standing conjecture proposed by Jorgensen and Pedersen and studied by many other authors.
\end{abstract}
\maketitle \tableofcontents
\section{Introduction}
In 1974, Fuglede \cite{Fug74} was studying a question of Segal on the existence of {\it commuting} extensions of the partial differential operators on domains of $\br^d$. Fuglede proved that the domains $\Omega$ for which such extensions exist are exactly those with the property that there exists an orthogonal basis for $L^2(\Omega)$, with Lebesgue measure, formed with exponential functions $\{e^{2\pi i\ip{\lambda}{x}} :\lambda\in\Lambda\}$ where $\Lambda$ is some discrete subset of $\br^d$. Such sets were later called {\it spectral sets} and $\Lambda$ was called a {\it spectrum} for $\Omega$.

\medskip

In the same paper, Fuglede proposed his famous conjecture that claims that the spectral sets are exactly those that tile $\br^d$ by some translations. The conjecture was later proved to be false in dimension 5 or higher, by Tao \cite{Tao04} and then in dimension 3 or higher \cite{MR2159781,MR2237932,MR2264214,MR2267631}. At this moment, the conjecture is still open in dimensions 1 and 2.

\medskip

In 1998, while working on the Fuglede conjecture, Jorgensen and Pedersen \cite{JP98} asked a related question: what are the {\it measures} for which there exist orthogonal bases of exponential functions?

\medskip

Let $\mu$ be a compactly supported Borel probability measure on
${\br}^d$ and let $\langle\cdot,\cdot\rangle$ denote the standard inner product on $\br^d$. The measure $\mu$ is called a {\it spectral measure} if there
exists a countable set $\Lambda\subset {\mathbb R}^d$, called {\it spectrum} of the measure $\mu$, such  that the collection of exponential functions
$E(\Lambda): = \{e^{2\pi i \langle\lambda,x\rangle}:
\lambda\in\Lambda\}$ forms an orthonormal basis for $L^2(\mu)$. We define the Fourier transform of $\mu$  to be
 $$
\widehat{\mu}(\xi)= \int e^{-2\pi i \langle\xi,x\rangle}d\mu(x).
 $$

\medskip

In \cite{JP98}, Jorgensen and Pedersen made a surprising discovery: they constructed the first example of a singular, non-atomic spectral measure. The measure is the Hausdorff measure associated to a Cantor set, where the scaling factor is 4 and the digits are 0 and 2. They also proved that the usual Middle Third Cantor measure is non-spectral. Also, Strichartz proved in \cite{MR2279556} that the Fourier series associated to such spectral fractal measures can have much better convergence properties than their classical counterparts on the unit interval: Fourier series of continuous functions converge uniformly, Fourier series of $L^p$-functions converge in the $L^p$-norm.
 Since Jorgensen and Pedersen's discovery, many other examples of singular measures have been constructed, and various classes of fractal measures have been analyzed \cite[and references therein]{JP98,LaWa02,Str98,Str00,DJ06,MR3163581,MR3273183,MR3302160,MR3318656}. All these constructions have used the central idea of Hadamard matrices and {\it Hadamard triples} to construct the spectral singular measures by infinitely many iterations. It has been conjectured since Jorgensen and Pedersen's discovery  that {\it all Hadamard triples will generate spectral self-affine measures.} Let us recall all the necessary definitions below:

\begin{definition}\label{hada}
Let $R\in M_d({\mathbb Z})$ be an $d\times d$ expansive matrix (all eigenvalues have modulus strictly greater than 1) with integer entries. Let $B, L\subset{\mathbb Z}^d $ be  finite sets of integer vectors with $N:= \#B=\#L$ ($\#$ denotes the cardinality). We say that the system $(R,B,L)$ forms a {\it Hadamard triple} (or $(R^{-1}B, L)$ forms a {\it compatible pair} in \cite{LaWa02} ) if the matrix
\begin{equation}\label{Hadamard triples}
H=\frac{1}{\sqrt{N}}\left[e^{2\pi i \langle R^{-1}b,\ell\rangle}\right]_{\ell\in L, b\in B}
\end{equation}
is unitary, i.e., $H^*H = I$.
\end{definition}

%

\begin{definition}\label{defifs}

For a given expansive $d\times d$ integer matrix $R$ and a finite set of integer vectors $B$ with $\#B =: N$, we define the {\it affine iterated function system} (IFS) $\tau_b(x) = R^{-1}(x+b)$, $x\in \br^d, b\in B$. The {\it self-affine measure} (with equal weights) is the unique probability measure $\mu = \mu(R,B)$ satisfying
\begin{equation}\label{self-affine}
\mu(E) = \sum_{b\in B} \frac1N \mu (\tau_b^{-1} (E)),\mbox{ for all Borel subsets $E$ of $\br^d$.}
\end{equation}
This measure is supported on the {\it attractor} $T(R,B)$ which is the unique compact set that satisfies
$$
T(R,B)= \bigcup_{b\in B} \tau_b(T(R,B)).
$$
The set $T(R,B)$ is also called the {\it self-affine set} associated with the IFS. One can refer to \cite{Hut81} and \cite{Fal97} for a detailed exposition of the theory of iterated function systems. We say that $\mu = \mu(R,B)$ satisfies the {\it no overlap condition} if
$$
\mu(\tau_{b}(T(R,B))\cap \tau_{b'}(T(R,B)))=0, \ \forall b\neq b'\in B.
$$
We say that $B$ is {\it a simple digit set for $R$}  if distinct elements of $B$ are not congruent $(\mod R(\bz^d))$.
\end{definition}

It is well known that if $(R,B,L)$ forms a Hadamard triple, then $B$ must be a simple digit set for $R$ and $L$ must be a simple digit set for $R^T$.   Furthermore, we need only consider equal-weight measures since, in the case when the weights are not equal, the self-affine measures cannot admit any spectrum, by the no overlap condition (see Theorem \ref{thDL15} below) and \cite[Theorem 1.5]{DL14}.

\medskip

In this paper we prove that the conjecture proposed by Jorgensen and Pedersen is valid:

\begin{theorem}\label{thmain}
Let $(R,B,L)$ be a Hadamard triple. Then the self-affine measure $\mu(R,B)$ is spectral.
\end{theorem}


\medskip

In dimension 1, Theorem  \ref{thmain} was first proved by Laba and Wang \cite{LaWa02} and refined in \cite{DJ06}. The situation becomes more complicated when $d>1$. Dutkay and Jorgensen showed that the conjecture  is true if $(R,B,L)$ satisfies a technical condition called {\it reducibility condition} \cite{DJ07d}. There are some other additional assumptions proposed by Strichartz guaranteeing Theorem \ref{thmain} is true \cite{Str98,Str00}. Some low-dimensional special cases were also considered by Li \cite{MR3163581,MR3302160}. In \cite{DL15}, we introduced the following set
\begin{equation}\label{Z}
{\mathcal Z} = \{\xi: \widehat{\mu}(\xi+k)=0, \ \mbox{for all} \ k\in{\mathbb Z}^d\}
\end{equation}
and proved

\begin{theorem}\label{thDL15} \cite[Theorem 1.7 and 1.8]{DL15}
Let $(R,B,L)$ be a Hadamard triple and $\mu=\mu(R,B)$ be the associated equal weight self-affine measure. Then we have:

\medskip

(i) $\mu$ has the no-overlap condition.

\medskip

(ii) Suppose furthermore that, ${\mathcal Z} = \emptyset$, then $\mu$ is a spectral measure with a spectrum in ${\mathbb Z}^d$.

\end{theorem}

\medskip

The complete resolution of Theorem \ref{thmain} points towards the case ${\mathcal Z}\neq \emptyset$.
 It was found that there exist spectral self-affine measures with ${\mathcal Z}\neq \emptyset$ \cite[Example 5.4]{DL15}. To prove Theorem \ref{thmain}, our strategy is to first show that in the case when ${\mathcal Z}\neq \emptyset$ the digit set $B$ will be reduced a {\it quasi product-form} structure (Section 3). This requires an analysis of ${\mathcal Z}$ as an invariant set of some dynamical system, and we use the techniques in \cite{CCR} (Section 2). Our methods are also similar to the ones used in \cite{LW2}. However, as $B$ is not a complete set of representatives (mod $R({\mathbb Z}^d)$) (as it is in \cite{LW2}), several additional adjustments will be needed.  From the quasi-product form structure obtained, we construct the spectrum directly by induction on the dimension $d$ (Section 4).

\medskip

%
%
%
%
%

\section{Preliminaries}
We first discuss some preliminary reduction that we can perform in order to prove our main theorem.

\begin{definition}\label{defconj}
Let $R_1,R_2$ be $d\times d$ integer matrices, and the finite sets $B_1,B_2,L_1,L_2$ be in $\bz^d$.
We say that two triples $(R_1, B_1, L_1)$ and $(R_2, B_2, L_2)$ are {\it conjugate} (through the matrix $M$) if there exists an integer matrix $M$ such that $R_2 = MR_1M^{-1}$, $B_2 = MB_1$ and
$L_2 = (M^T)^{-1}L_1$.
\end{definition}

\medskip

\begin{proposition}\label{prconj}
Suppose that $(R_1, B_1, L_1)$ and $(R_2, B_2, L_2)$ are two conjugate triples, through the matrix $M$. Then
\medskip

(i) If $(R_1,B_1,L_1)$ is a Hadamard triple then so is $(R_2,B_2,L_2)$.

\medskip

(ii)The measure $\mu(R_1,B_1)$ is spectral with spectrum $\Lambda$ if and only if $\mu(R_2,B_2)$ is spectral with spectrum $(M^T)^{-1}\Lambda$.
\end{proposition}

\begin{proof}
The proof follows from some simple computations, see e.g. \cite[Proposition 3.4]{DJ07d}.
\end{proof}

\medskip

Fora given integral expanding matrix $R$ and  a simple digit set $B$ for $R$. We let
 \begin{equation}\label{B_n}
B_n := B+RB+R^{2}B+...+R^{n-1}B = \left\{\sum_{j=0}^{n-1}R^jb_j: b_{j}\in B\right\}.
\end{equation}
	
  We define ${\mathbb Z}[R,B] $ to be the smallest $R$-invariant lattice containing all $B_n$ (invariant means $R(\bz[R,B])$$\subset\bz[R,B]$).   By Proposition \ref{przrb} below, to prove Theorem \ref{thmain}, there is no loss of generality if we assume that ${\mathbb Z}[R,B] = {\mathbb Z}^d$.

\begin{proposition}\label{przrb}
If the lattice $\bz[R,B]$ is not full-rank, then the dimension can be reduced; more precisely, there exists $1\leq r<d$ and a unimodular matrix $M\in GL(n,\bz)$ such that $M(B)\subset \bz^r\times\{0\}$ and
\begin{equation}
MRM^{-1}=\begin{bmatrix}
A_1& C\\
0& A_2
\end{bmatrix}
\label{eqzrb1}
\end{equation}
where $A_1\in M_r(\bz)$, $C\in M_{r,d-r}(\bz), A_2\in M_{d-r}(\bz)$. In addition,  $M(T(R,B))\subset \br^r\times \{0\}$ and the Hadamard triple $(R,B,L)$ is conjugate to the Hadamard triple $(MRM^{-1},MB,(M^T)^{-1}L)$, which is a triple of lower dimension.

\medskip

If the lattice $\bz[R,B]$ is full rank but not $\bz^d$, then  the system $(R,B,L)$ is conjugate to one $(\tilde R,\tilde B,\tilde L)$ for which $\bz[\tilde R,\tilde B]=\bz^d$. Moreover, $M$ is given by ${\mathbb Z}[R,B] = M({\mathbb Z}^d)$. 
\end{proposition}

\begin{proof}
See Proposition 4.1 in \cite{DL15}
\end{proof}

\medskip

In the following, we introduce the main technique that will be used. We start with the following definition.

\begin{definition}\label{definv}
Let $u\geq0$ be an entire function on $\br^d$, i.e., real analytic on $\br^d$. Let $\cj L$ be a simple digit set for $R^T$. Suppose that
\begin{equation}
\sum_{l\in \cj L}u((R^T)^{-1}(x+l))>0,\quad(x\in\br^d)
\label{eqdefinv1}
\end{equation}

A closed set $K$ in $\br^d$ is called {\it $u$-invariant (with respect to the system $(u, R^T,\cj L)$)} if, for all $x\in K$ and all $\ell\in\cj L$
$$
 u\left(((R^T)^{-1}(x+\ell)\right)>0 \ \Longrightarrow \ (R^{T})^{-1}(x+\ell)\in K.
 $$
We say that the transition, using $\ell$, from $x$ to $\tau_{\ell}(x)$ is possible, if $\ell\in \cj L$ and $u\left((R^{T})^{-1}(x+\ell)\right)>0$. We say that $K$ is ${\mathbb Z}^d$-periodic if $K+n=K$ for all $n\in{\mathbb Z}^d$.
\end{definition}

\medskip

We say that a subspace $W$ of ${\mathbb R}^d$ is a {\it rational subspace} if $W$ has a basis of vectors with rational components. The following theorem follows from Proposition 2.5, Theorem 2.8 and Theorem 3.3 in \cite{CCR}.
\begin{theorem}\label{thccr}
Let $\cj L$ be a complete set of representatives $(\mod R^T(\bz^d))$.
Let $u\geq0$ be an entire function on $\br^d$ and let $K$ be  a closed $u$-invariant $\bz^d$-periodic set different from $\br^d$. Suppose in addition that $g$ is an entire function which is zero on $K$. Then

\medskip

(i) there exists a point $x_0\in\br^d$, such that $(R^T)^mx_0\equiv x_0(\mod \ \bz^d)$ for some integer $m\geq1$, and

 \medskip

 (ii) a proper rational subspace $W$ (may equal $\{0\}$)  such that $R^{T}(W) = W$ and the union
$$
\mathcal S=\bigcup_{k=0}^{m-1}((R^T)^kx_0+W+\bz^d)
$$
is invariant and $g$ is zero on $\mathcal S$.

Moreover, all possible transitions from a point in $(R^T)^kx_0+W+\bz^d$, $1\leq k\leq m$, lead to a point in $(R^T)^{k-1}x_0+W+\bz^d$.
\end{theorem}

\medskip

Let $(R,B,L)$ be a Hadamard triple and we aim to apply Theorem \ref{thccr} to our set ${\mathcal Z}$ in \eqref{Z}. We define the function

\begin{equation}
u_B(x)=\left|\frac{1}{N}\sum_{b\in B}e^{2\pi i\ip{b}{x}}\right|^2,\quad(x\in\br^d).
\label{eq1.12.1}
\end{equation}
Taking the Fourier transform of the invariance equation \eqref{self-affine}, we can compute explicitly the Fourier transform of $\mu = \mu(R,B)$ as
\begin{equation}
|\widehat\mu(\xi)|^2=u_B((R^T)^{-1}\xi)|\widehat\mu((R^T)^{-1}(\xi))|^2,\quad(x\in\br^d).
\label{eq1.12.2}
\end{equation}
Iterating \eqref{eq1.12.2}, we obtain
\begin{equation}
|\widehat\mu(x)|^2=\prod_{n=1}^\infty u_B((R^T)^{-n}x),\quad(x\in\br^d),
\label{eq1.12.3}
\end{equation}
and the convergence in the product is uniform on compact sets. See e.g. \cite{DJ07d}. It is well known that both $u_B$ and $|\widehat{\mu}|^2$ are entire functions on ${\mathbb R}^d$.

\begin{proposition}\label{pr1.13}
Suppose that $(R,B,L)$ forms a Hadamard triple
and $\bz[R,B]=\bz^d$. Let $\cj L$ be a complete set of representatives $(\mod R^T(\bz^d))$ containing $L$. Suppose that the set
$$\mathcal Z:=\left\{\xi\in \br^d : \widehat\mu (\xi+k)=0\mbox{ for all }k\in\bz^d\right\}$$
is non-empty. Then

 \medskip

 (i) $\mathcal Z$ is $u_B$-invariant.

 \medskip
 (ii) There exist a point $x_0\in\br^d$ such that $(R^T)^mx_0\equiv x_0(\mod (R^T)\bz^d)$, for some integer $m\geq 1$.

  \medskip
  (iii) There exists a proper rational subspace $W\neq\{0\}$ of $\br^d$ such that $R^{T}(W)=W$ and the union
$$\mathcal S=\bigcup_{k=0}^{m-1}((R^T)^kx_0+W+\bz^d)$$
is $u_B$-invariant and is contained in $\mathcal Z$.

\medskip
Moreover, all possible transitions  from a point in $(R^T)^kx_0+W+\bz^d$, $1\leq k\leq m$, lead to a point in $(R^T)^{k-1}x_0+W+\bz^d$.

\end{proposition}

\begin{proof}

We first prove  that $\mathcal Z$ is $u_B$-invariant. Take $x\in\mathcal Z$ and $\ell\in\cj L$ such that $u_B((R^T)^{-1}(x+\ell))>0$. Let $k\in\bz^d$. We have, with \eqref{eq1.12.2},
$$0=|\widehat\mu(x+\ell+R^Tk)|^2=u_B((R^T)^{-1}(x+\ell+R^Tk))|\widehat\mu((R^T)^{-1}(x+\ell+R^Tk))|^2$$$$=u_B((R^T)^{-1}(x+\ell))|\widehat\mu((R^T)^{-1}(x+\ell)+k)|^2.$$
Therefore, $\widehat\mu((R^T)^{-1}(x+\ell)+k)=0$ for all $k\in\bz^d$. So $(R^T)^{-1}(x+\ell)$ is in $\mathcal Z$, and this shows that $\mathcal Z$ is $u_B$-invariant.

Since $(R,B,L)$ form a Hadamard triple, by the Parseval identity, (see e.g. \cite{LaWa02,DJ07d}),
\begin{equation}\label{eqqmf}
\sum_{l\in L}u_B((R^T)^{-1}(x+l))=1,\quad(x\in\br^d),
\end{equation}
Hence,
\begin{equation}
\sum_{l\in \cj L}u_B((R^T)^{-1}(x+l))>0,\quad(x\in\br^d).
\label{eqqmf2}
\end{equation}

We can apply Theorem \ref{thccr} with $u=u_B$ and $g=\widehat\mu$ to obtain all other conclusions except the non-triviality of $W$. We now check that $W\neq\{0\}$. Suppose $W=\{0\}$. First we show that for $1\leq k\leq m$ there is a unique $\ell\in \cj L$ such that
$u_B((R^T)^{-1}((R^T)^kx_0+\ell)>0$. Equation \eqref{eqqmf2} shows that there exists at least one such $\ell$. Assume that we have two different $\ell$ and $\ell'$ in $\cj L$ with this property. Then the transitions are possible, so
\begin{equation}\label{eq1.12}
(R^T)^{-1}((R^T)^kx_0+\ell)\equiv (R^T)^{k-1}x_0\equiv (R^T)^{-1}((R^T)^kx_0+\ell') \ (\mod  \ (\bz^d)).
\end{equation}
But then $\ell\equiv \ell'(\mod R^T(\bz^d))$ and this is impossible since $\cj L$ is a complete set of representatives.

\medskip

By a translation, we can assume $0\in B$. From \eqref{eqqmf}, and since the elements in $L$ are distinct $(\mod R^T(\bz^d))$, we see that there is exactly one $\ell_k\in L$ such that $u_B((R^T)^{-1}((R^T)^kx_0+\ell_k))>0$. Therefore $u_B((R^T)^{-1}((R^T)^kx_0+\ell_k))=1$. But then, by \eqref{eq1.12}, $u_B((R^T)^{k-1}x_0)=1$. We have
$$
\left|\sum_{b\in B}e^{2\pi i \ip{b}{(R^{T})^{k-1}x_0}}\right| = N.
$$
 As $\#B = N$ and $0\in B$, we have equality in the triangle inequality, and we get that $e^{2\pi i \ip{b}{(R^T)^kx_0}}=1$ for all $b\in B$. Then $\ip{R^{k-1}b}{x_0}\in{\mathbb Z}$ for all $b\in B$, $1\leq k\leq m$. Because $(R^T)^m x_0\equiv x_0(\mod\bz^d)$, we get that $\ip{R^kb}{x_0}\in{\mathbb Z}$ for all $k\geq 0$ and thus
$$x_0\in\bz[R,B]^\perp:=\{x\in\br^d : \ip{\lambda}{x}\in\bz\mbox{ for all }\lambda\in \bz[R,B]\}.$$ Since $\bz[R,B]=\bz^d$, this means that $x_0\in \bz^d$. But $x_0\in\mathcal Z$, so $1=\widehat\mu(0)=\widehat\mu(x_0-x_0)=0$, which is a contradiction. This shows $W\neq\{0\}$.
\end{proof}

\medskip

As $W\neq \{0\}$, we can conjugate $R$ through some $M$ so that $(R,B,L)$ has a much more regular structure.

\begin{proposition}\label{pr1.14}
Suppose that $(R,B,L)$ forms a Hadamard triple and ${\mathbb Z}[R,B] = {\mathbb Z}^d$ and let $\mu = \mu(R,B)$ be the associated self-affine measure $\mu=\mu(R,B)$. Suppose that the set

$$\mathcal Z:=\left\{x\in \br^d : \widehat\mu (x+k)=0\mbox{ for all }k\in\bz^d\right\},$$
is non-empty. Then there exists an integer matrix $M$ with $\det M=1$ such that the following assertions hold:

\begin{enumerate}
	\item The matrix $\tilde R:=MRM^{-1}$ is of the form
	\begin{equation}
	\tilde R=\begin{bmatrix} R_1&0\\ C &R_2\end{bmatrix},
	\label{eq1.14.1}
	\end{equation}
	with $R_1\in M_r(\bz)$, $R_2\in M_{d-r}(\bz)$ expansive integer matrices and  $C\in M_{(d-r)\times r}(\bz)$.
	\item If $\tilde B=MB$ and $\tilde L=(M^T)^{-1}L$, then $(\tilde R, \tilde B,\tilde L)$ is a Hadamard triple.
	\item The measure $\mu(R,B)$ is spectral with spectrum $\Lambda$ if and only if the measure $\mu(\tilde R,\tilde B)$ is spectral with spectrum $(M^T)^{-1}\Lambda$.
	\item There exists $y_0\in\br^{d-r}$ such that $(R_2^T)^my_0\equiv y_0(\mod (R_2^T)\bz^d)$ for some integer $m\geq 1$ such that the union
	$$\tilde {\mathcal S}=\bigcup_{k=0}^{m-1}(\br^r\times \{(R_2^T)^ky_0\}+\bz^d)$$
	is contained in the set
	$$\tilde{\mathcal Z}:=\left\{x\in \br^d : \widehat{\tilde \mu} (x+k)=0\mbox{ for all }k\in\bz^d\right\},$$
	where $\tilde \mu=\mu(\tilde R,\tilde B)$. The set $\tilde{\mathcal S}$ is invariant (with respect to the system $(u_{\tilde B},\tilde R^T,\tilde {\cj L})$, where $\tilde {\cj L}$ is a complete set of representatives $(\mod \tilde R^T\bz^d)$. In addition, all possible transitions from a point in $\br^r\times \{(R_2^T)^ky_0\}+\bz^d$, $1\leq k\leq m$ leads to a point in $\br^r\times \{(R_2^T)^{k-1}y_0\}+\bz^d$.
\end{enumerate}
\end{proposition}

\begin{proof}
We use Proposition \ref{pr1.13} and we have $x_0$ and a rational subspace $W\neq\{0\}$ invariant for $R$ with all the mentioned properties. By \cite[Theorem 4.1 and Corollary 4.3b]{Sch86}, there exists an integer matrix $M$ with determinant 1 such that $MV=\br^r\times \{0\}$. The rest follows from Proposition \ref{pr1.13}, by conjugation, $y_0$ is the second component of $Mx_0$.

\end{proof}

\section{The quasi-product form}

From now on, we assume that $(R,B,L)$ satisfies all the properties of $(\tilde R,\tilde B,\tilde L)$ in Proposition \ref{pr1.14}. In this section, we will prove that if ${\mathcal Z}\neq \emptyset$, the Hadamard triple will be conjugate to a quasi-product form structure.

\medskip

We first introduce the following notations.

\begin{definition}\label{def1.15}
For a vector $x\in \br^d$, we write it as $x=(x^{(1)},x^{(2)})^T$ with $x^{(1)}\in \br^r$ and $x^{(2)}\in\br^{d-r}$. We denote by $\pi_1(x)=x^{(1)}$, $\pi_2(x)=x^{(2)}$. For a subset $A$ of $\br^d$, and $x_1\in\br^r$, $x_2\in \br^{d-r}$, we denote by
$$A_2({x_1}):=\{y\in\br^{d-r} :(x_1,y)^T\in A\},\quad A_1({x_2}):=\{x\in\br^r : (x,x_2)^T\in A\}.$$
\end{definition}

We also make a note on the notation. Throughout the rest of the paper,  we use $A\times B$ to denote the Cartesian product of $A$ and $B$ so that $A\times B = \{(a,b):a\in A, b\in B\}$.  Our main theorem in this section is as follows:

\begin{theorem}\label{th_quasi}
Suppose that $$R=\begin{bmatrix}
R_1& 0\\ C& R_2
\end{bmatrix},
$$$(R,B,L_0)$ is a Hadamard triple and $\mu = \mu(R,B)$ is the associated self-affine measure and ${\mathcal Z}\neq \emptyset$. Then the set $B$ has the following quasi-product form:
\begin{equation}
B=\left\{(u_i,v_i+Qc_{i,j})^T : 1\leq i\leq N_1, 1\leq j\leq|\det R_2|\right\},
\label{eq1.18.1}
\end{equation}
where
 \begin{enumerate}\item $N_1 = N/|\det R_2|$, \item $Q$ is a $(d-r)\times (d-r)$ integer matrix with $|\det Q|\geq 2$ and $R_2Q=Q\tilde R_2$ for some $(d-r)\times(d-r)$ integer matrix $\widetilde{R_2}$, \item the set $\{Qc_{i,j}: 1\leq j\leq |\det R_2|\}$ is a complete set of representatives $(\mod R_2(\bz^{d-r}))$, for all $1\leq i\leq N_1$.
\end{enumerate}
\medskip

Moreover, one can find some $L\equiv L_0 (\mod \ R^T(\mathbb Z^d))$ so that $(R,B,L)$ is a Hadamard triple and $(R_1,\pi_1(B), L_1(\ell_2))$ and $(R_2,B_2(b_1),\pi_2(L))$ are Hadamard triples on ${\mathbb R}^r$ and ${\mathbb R}^{d-r}$ respectively.
\end{theorem}

The following lemma allows us to find a representative $L$ with certain injectivity property.

\begin{lemma}\label{lem1.15-}
Suppose that the Hadamard triple $(R,B,L)$ satisfies the properties (i) and (iv)  in Proposition \ref{pr1.14}. Then there exists set $L'$ and a complete set of representatives $\cj L'$ such that $(R,B,L')$ is a Hadamard triple and the following property holds: \begin{equation}\label{eq1.15.0}
\ell,\ell'\in L' \ \mbox{(or $\in \cj L'$) and} \ \pi_2(\ell)\equiv \pi_2(\ell') \ \left(\mod \  R_2^T(\bz^{d-r})\right) \ \ \Longrightarrow \  \ \pi_2(\ell)=\pi_2(\ell').
\end{equation}
\end{lemma}

\begin{proof}
We note that if $L\equiv L'$ $(\mod R^T({\bz^{d}}))$, then $(R,B,L')$ is a Hadamard triple since $\ip{R^{-1}b}{R^{T}m}$ $\in{\mathbb Z}$ for any $m\in{\mathbb Z}^d$. Now, let $\ell=(\ell_1,\ell_2)^T, l'=(\ell_1',\ell_2')^T$ be in $L$ (or $\cj L$) such that $\ell_2\equiv \ell_2'(\mod R_2^T(\bz^{d-r}))$. We replace $\ell'$ by
$$\ell''=\ell'+R^T(0,(R_2^T)^{-1}(\ell_2-\ell_2'))^T\in\bz^d.$$
 Then $\ell''\equiv \ell' (\mod R^T({\bz^{d}}))$ and the Hadamard property for $L$ is preserved, the new set $\cj L$ is a complete set of representatives $(\mod R^T(\bz^d))$ and $\pi_2(\ell'')=l_2$. Repeating this procedure, we obtain our lemma.
\end{proof}

\medskip

To simplify the notation, in what follows we relabel $L'$ by $L$ and $\cj L'$ by $\cj L$ so that $L$ and $\cj L$ possess property (\ref{eq1.15.0}). We prove some lemmas for the proof of Theorem \ref{th_quasi}.

\begin{lemma}\label{lem1.15}
Suppose that the Hadamard triple $(R,B,L)$ satisfies the properties (i) and (iv)  in Proposition \ref{pr1.14}. Then
\begin{enumerate}
%
%

\item For every $b_1\in \pi_1(B)$ and $b_2\neq b_2'$ in $B_2(b_1)$,
\begin{equation}
\sum_{\ell_2\in \pi_2(L)}\# L_1(\ell_2)e^{2\pi i \ip{R_2^{-1}(b_2-b_2')}{\ell_2}}=0.
\label{eq1.15.1}
\end{equation}
Also, for all $b_1\in \pi_1(B)$,  $\#B_2(b_1)\leq \#\pi_2(L)$  and the elements in $B_2(b_1)$ are not congruent $\mod R_2(\bz^{d-r})$.

\medskip

\item For every $\ell_2\in \pi_2(L)$ and $\ell_1\neq \ell_1'$ in $L_1(\ell_2)$,
\begin{equation}
\sum_{b_1\in\pi_1(B)}\# B_2(b_1)e^{2\pi i \ip{R_1^{-1}b_1}{(\ell_1-\ell_1')}}=0.
\label{eq1.15.2}
\end{equation}
Also, for all $\ell_2\in\pi_2(L)$, $\#L_1(\ell_2)\leq \#\pi_1(B)$ and the elements in $L_1(\ell_2)$ are not congruent $\left(\mod \  R_1^T(\bz^{r})\right)$.

\medskip
\item The set $\pi_2(\cj L)$ is a complete set of representatives $(\mod  \ R_2^T(\bz^{d-r}))$ and, for every $\ell_2\in \pi_2(\cj L)$, the set $\cj L_1(\ell_2)$ is a complete set of representatives $(\mod \  R_1^T(\bz^r))$.
\end{enumerate}
\end{lemma}

\begin{proof}
We first prove (i).
Take $b_1\in\pi_1(B)$ and $b_2\neq b_2'$ in $B_2(b_1)$, from the mutual orthogonality, we have
$$
0=\sum_{\ell_2\in\pi_2(L)}\sum_{\ell_1\in L_1(\ell_2)}e^{2\pi i \ip{(R^T)^{-1}(0,b_2-b_2')^T}{ (\ell_1,\ell_2)^T}}=\sum_{\ell_2\in\pi_2(L)}\sum_{\ell_1\in L_1(\ell_2)}e^{2\pi i \ip{(R_2^T)^{-1}(b_2-b_2')}{ \ell_2}}
$$
$$
=\sum_{\ell_2\in\pi_2(L)}\#L_1(\ell_2)e^{2\pi i \ip{(R_2^T)^{-1}(b_2-b_2')}{\ell_2}}.
$$
This shows \eqref{eq1.15.1} and that the rows of the matrix
$$\left(\sqrt{\# L_1(\ell_2)}e^{2\pi i \ip{(R_2^T)^{-1}b_2}{\ell_2}}\right)_{\ell_2\in \pi_2(L),b_2\in B_2(b_1)}$$
are orthogonal. Therefore $\#B_2(b_1)\leq \#\pi_2(L)$ for all $b_1\in \pi_1(B)$. Equation \eqref{eq1.15.1} implies that the elements in $B_2(b_1)$ cannot be congruent $\mod R_2^T(\bz^{d-r})$

\medskip

(ii) follows from an analogous computation.

\medskip

For (iii), the elements in $\pi_2(\cj L)$ are not congruent $(\mod R_2^T(\bz^{d-r}))$, by the construction in Lemma \ref{lem1.15-}. If $\ell_2\in \pi_2(\cj L)$ and $\ell_1,\ell_1'\in \cj L_1(l_2)$ are congruent $(\mod R_1^T(\bz^r))$ then $(\ell_1,\ell_2)^T\equiv (\ell_1',\ell_2)^T$$(\mod \ R^T(\bz^d))$. Thus, $\ell_1=\ell_1'$, as $\cj L$ is a complete set of representatives of $R^T({\mathbb Z}^d)$. From these, we have $\#\pi_2(L)\leq |\det R_2|$ and, for all $\ell_2\in \pi_2(\cj L)$, $\#L_1(\ell_2)\leq |\det R_1|$. Since
$$|\det R|=|\det R_1||\det R_2|\geq \sum_{\ell_2\in \pi_2(\cj L)}\#\cj L_1(\ell_2)=\#\cj L=|\det R|,$$
we must have equalities in all inequalities and we get that the sets are indeed {\it complete} sets of representatives.
\end{proof}

\medskip

\begin{lemma}\label{lem1.16}
Let $1\leq j\leq m$. If the  the transition from $(x,(R_2^T)^jy_0)^T$ is possible with the digit $\ell\in \cj L$, then $\pi_2(\ell)=0$.
\end{lemma}

\begin{proof}
If the transition is possible with digit $\ell=(\ell_1,\ell_2)^T$, then, by Proposition \ref{pr1.14},  $$(R^T)^{-1}((x,(R_2^T)^jy_0)^T+(\ell_1,\ell_2)^T)\equiv (y,(R_2^T)^{j-1}y_0)^T(\mod\bz^d),$$ for some $y\in\br^r$, and therefore $(R_2^T)^{-1}\ell_2\equiv 0 \ (\mod \ \bz^{d-r})$, so $\ell_2\equiv 0 \ (\mod R_2^T(\bz^{d-r}))$. By Lemma \ref{lem1.15-}, $\ell_2=0$.
\end{proof}

\medskip

\begin{lemma}
Let $y_j:=(R_2^T)^jy_0$, $1\leq j\leq m$. Then, for all $x\in\br^r$ and all $\ell=(\ell_1,\ell_2)\in \cj L$ with $\pi_2(\ell)=\ell_2\neq 0$, we have that
\begin{equation}
\sum_{b_2\in B_2(b_1)}e^{2\pi i \ip{b_2}{ (R_2^T)^{-1}(y_j+ \ell_2)}}=0.
\label{eq1.17.1}
\end{equation}
\end{lemma}
\begin{proof}
We have that $(R^T)^{-1}$ is of the form $$(R^T)^{-1}=\begin{bmatrix}
(R_1^T)^{-1}& D\\ 0& (R_2^T)^{-1}
\end{bmatrix}.
$$
Then, for all $x\in\br^r$ and all $\ell=(\ell_1,\ell_2)\in \cj L$ with $\pi_2(\ell)=\ell_2\neq 0$, we have that $u_B((R^T)^{-1}((x,y_j)^T+(\ell_1,\ell_2)^T))=0$, because such transitions are not possible by Lemma \ref{lem1.16}. Then
$$
\sum_{b_1\in \pi_1(B)}\sum_{b_2\in B_2(b_1)}e^{2\pi i (\ip{b_1}{ R_1^{-1}(x+l_1)+D(y_1+l_2)}+\ip{b_2}{ (R_2^T)^{-1}(y_j+l_2)})}=0.
$$
Since $x$ is arbitrary,  it follows that
$$
\sum_{b_1\in\pi_1(B)}e^{2\pi i \ip{b_1}{x}}\sum_{b_2\in B_2(b_1)}e^{2\pi i \ip{b_2}{(R_2^T)^{-1}(y_j+l_2)}}=0 \ \mbox{for all}  \ x\in{\mathbb R}^d.
$$
Therefore, by linear independence of exponential functions, we obtain \eqref{eq1.17.1}.
\end{proof}

\medskip

\begin{lemma}\label{lem1.18}
For every $b_1\in\pi_1(B)$, the set $B_2(b_1)$ is a complete set of representatives $(\mod R_2(\bz^{d-r}))$. Therefore $\#B_2(b_1)=|\det R_2|=\#\pi_2(L)$ and
$(R_2,B_2(b_1),\pi_2(L))$ is a Hadamard triple. Also, for every $\ell_2\in \pi_2(L)$, $\#L_1(\ell_2)=\#\pi_1(B)=\frac{N}{|\det R_2|}=:N_1$ and
$(R_1,\pi_1(B),L_1(\ell_2))$ is a Hadamard triple. 
\end{lemma}

\begin{proof}
Let $b_1\in\pi_1(B)$.  We know from Lemma \ref{lem1.15}(i) that the elements of $B_2(b_1)$ are not congruent $(\mod \ R_2(\bz^{d-r}))$. We can identify $B_2(b_1)$ with a subset of the group $\bz^{d-r}/R_2(\bz^{d-r})$. The dual group is $\bz^{d-r}/R_2^T(\bz^{d-r})$ which we can identify with $\pi_2(\cj L)$.
For a function $f$ on $\bz^{d-r}/R_2(\bz^{d-r})$, the Fourier transform is
$$\hat f(\ell_2)=\frac{1}{\sqrt{|\det R_2|}}\sum_{b_2\in \bz^{d-r}/R_2(\bz^{d-r})} f(b_2)e^{-2\pi i\ip{b_2}{(R_2^T)^{-1}\ell_2}},\quad (\ell_2\in \pi_2(L_2)).$$

Let $1\leq j\leq m$. Consider the function
\begin{equation}\label{f(b)}
f(b_2)=\left\{\begin{array}{cc}e^{-2\pi i\ip{b_2}{(R_2^T)^{-1}y_j}},&\mbox{ if }b_2\in B_2(b_1)\\
0,&\mbox{ if }b_2\in (\bz^{d-r}/R_2(\bz^{d-r}))\setminus B_2(b_1).\end{array}\right.
\end{equation}
Then equation \eqref{eq1.17.1} shows that $\hat f(\ell_2)=0$ for $\ell_2\in \pi_2(\cj L)$, $\ell_2\neq 0$. Thus $\hat f=c\cdot\chi_{0}$ for some constant $c$ and by $\hat{f}(0) = c$,
\begin{equation}\label{c}
c=\frac{1}{\sqrt{|\det R_2|}}\sum_{b_2\in B_2(b_1)}e^{-2\pi i \ip{b_2}{(R_2^T)^{-1}(y_j)}}.
\end{equation}

\medskip

 Now we apply the inverse Fourier transform and we get
$$f(b_2)=\frac{1}{\sqrt{|\det R_2|}}\sum_{\ell_2\in \pi_2(\cj L)}c\chi_{0}(\ell_2)e^{2\pi i \ip{b_2}{ (R_2^T)^{-1}\ell_2}}=\frac{c}{\sqrt{|\det R_2|}}.$$
So $f(b_2)$ is constant and therefore $B_2(b_1)=\bz^{d-r}/R_2(\bz^{d-r})$, which means that $B_2(b_1)$ is a complete set of representatives and $\#B_2(b_1)=|\det R_2|$. Since the elements in $\#\pi_2(L)$ are not congruent $(\mod R_2^T\bz^{d-r})$, we get that $\# \pi_2(L)\leq |\det R_2|$, and with Lemma \ref{lem1.15} (i), it follows that $\#\pi_2(L)=\#B_2(b_1)=|\det R_2|$. In particular, $\pi_2(L)$ is a complete set of representatives $(\mod R_2^T(\bz^{d-r}))$, so $(R_2,B_2(b_1),\pi_1(L))$ form a Hadamard triple.

\medskip

Since $\sum_{b_1\in \pi_1(B)}\#B_2(b_1)=N$, we get that $\#\pi_1(B)=N/|\det R_2|$. With Lemma \ref{lem1.15}(ii), we have
$$N=\sum_{\ell_2\in\pi_2(L)}\#L_1(\ell_2)\leq \#\pi_2(L)\pi_1(B)=N.$$
Therefore we have equality in all inequalities so $\#L_1(\ell_2)=\#\pi_1(B)=N/|\det R_2|$. Then \eqref{eq1.15.2} shows that $(R_1,\pi_1(B),L_1(\ell_2))$ is a Hadamard triple for all $\ell_2\in \pi_2(L)$.
\end{proof}

\medskip

\begin{proof}[Proof of Theorem \ref{th_quasi}] By Lemma \ref{lem1.18}, we know that $B$ must have the form
 $$
 \bigcup_{b_1\in\pi_1(B)} \{b_1\}\times B_2(b_1)
 $$
 where $\#\pi_1(B) = N_1$ and $B_2(b_1)$ is a set of complete representative (mod $R_2^T({\mathbb Z}^{d-r})$). By enumerating elements $\pi_1(B)  = \{u_1,...,u_{N_1}\}$ and $B_2(u_{i}) = \{d_{i,1},...,d_{i,|\det R_2|}\}$, we can write
 $$
 B =\left\{(u_i,d_{i,j})^T : 1\leq i\leq N_1, 1\leq j\leq|\det R_2|\right\}.
 $$
It suffices to show $d_{i,j}$ are given by $v_i+Qc_{i,j}$ where $Q$ has the properties (ii) and (iii) in the theorem. From the equation (\ref{f(b)}) and the fact that $f$ is a constant,
we have, for $b_2\in B_2(b_1)$ and $b_1\in\pi_1(B)$,
$$
e^{-2\pi i \ip{b_2}{(R_2^T)^{-1}y_j}}=f(b_2)=\frac{c}{\sqrt{|\det R_2|}},
$$
which implies (from (\ref{c})) that
$$
\frac{1}{|\det R_2|}\sum_{b_2'\in B_2(b_1)}e^{2\pi i \ip{(b_2-b_2')}{(R_2^T)^{-1}y_j}}=1.
$$
By applying the triangle inequality to the sum above, we see that we must have
$$
e^{2\pi i\ip{b_2-b_2'}{(R_2^T)^{-1}y_j}}=1,
$$
which means
\begin{equation}\label{b_2-b_2'}
\ip{b_2-b_2'}{(R_2^T)^{-1} y_j}\in\bz \  \mbox{for all} \ b_2,b_2'\in B_2(b_1), b_1\in \pi_1(B), 1\leq j\leq m.
\end{equation} Here we recall that  $y_j = (R_2^T)^{j}y_0$.

\medskip

Define now the lattice
$$\Gamma:=\{x\in\bz^{d-r} : \ip{x}{(R_2^T)^{-1}y_j}\in{\mathbb Z},  \ \forall \ 1\leq j\leq m\}.$$

\medskip

We first claim that the lattice $\Gamma$ is of full-rank. Indeed, since $(R_2^T)^my_j\equiv y_j(\mod \bz^{d-r})$, it follows that all the points $(R_2^T)^{-1}(y_j)$ have only rational components. Let $\tilde m$ be a common multiple for all the denominators of all the components of the vectors $(R_2^T)^{-1}(y_j)$, $1\leq j\leq m$. If $\{e_i\}$ are the canonical vectors in $\br^{d-r}$, then $\ip{\tilde me_i}{ (R_2^T)^{-1}(y_j)}\in\bz$ so $\tilde m e_i\in\Gamma$, and thus $\Gamma$ is full-rank.

\medskip

Next we prove that $\Gamma$ is a proper sublattice of $\bz^{d-r}$. The vectors $y_j$ are not in $\bz^{d-r}$ because $\widehat\mu((0,y_j)^T+k)=0$ for all $k\in\bz^d$, and that would contradict the fact that $\widehat\mu(0)=1$. This implies that the vectors $(R_2^T)^{-1}(y_j)$ are not in $\bz^{d-r}$ so $\Gamma$ is a proper sublattice of $\bz^{d-r}$.

\medskip

Since $\Gamma$ is a full-rank lattice in $\bz^{d-r}$, there exists an invertible matrix with integer entries $Q$ such that $\Gamma=Q\bz^{d-r}$, and since $\Gamma$ is a proper sublattice, it follows that $|\det Q|>1$ so $|\det Q|\geq 2$.
In addition, we know from (\ref{b_2-b_2'}) that, for all $u_i\in \pi_1(B)$ and $d_{i,j},d_{i,j'}\in B_2(u_i)$, $d_{i,j}-d_{i,j'}\in\Gamma$. Therefore, if we fix an element $v_i=d_{i,j_0}\in B_2(a_i)$, then all the elements in $B_2(a_i)$ are of the form $d_{i,j}=v_i+Qc_{i,j}$ for some $c_{i,j}\in\bz^{d-r}$. The fact that $B_2(a_i)$ is a complete set of representatives $(\mod R_2\bz^{d-r})$ (Lemma \ref{lem1.18}) implies that  the set of the corresponding elements $Qc_{i,j}$ is also a complete set of representatives $(\mod R_2\bz^{d-r})$. This shows (iii).

\medskip

It remains to show $R_2Q = Q\widetilde{R_2}$ for some for some $(d-r)\times(d-r)$ integer matrix $\widetilde{R_2}$.  Indeed, if $x\in\Gamma$, and $0\leq j\leq m-1$, then
$\ip{R_2x}{(R_2^T)^jy_0}=\ip{x}{(R_2^T)^{j+1}y_0}\in\bz$, since $(R_2^T)^my_0\equiv y_0 \ (\mod\bz^{d-r})$. So $R_2x\in\Gamma$. Then, for the canonical vectors $e_i$, there exist $\tilde e_i\in\bz^{d-r}$ such that $R_2Qe_i=Q\widetilde{e_i}$. Let $\widetilde{R_2}$ be the matrix with columns $\widetilde{e_i}$. Then $R_2Q=Q\widetilde{R_2}$.

\medskip

Finally, by choosing $L$ with the property Lemma \ref{lem1.15-}, the Hadamard triple properties of both  $(R_1,\pi_1(B), L_1(\ell_2))$ and $(R_2,B_2(b_1),\pi_1(L))$ on ${\mathbb R}^r$ and ${\mathbb R}^{d-r}$ respectively are direct consequences of Lemma \ref{lem1.18}.
\end{proof}
\medskip

\section{Proof of the theorem}
In the last section, we will prove our main theorem. We first need to study the spectral property of the quasi-product form.
Suppose now the pair $(R,B)$ is in the quasi-product form
\begin{equation}\label{R_4.1}R=\begin{bmatrix}
R_1&0\\
C&R_2
\end{bmatrix}\end{equation}
\begin{equation}
B=\left\{(u_i,d_{i,j})^T : 1\leq i\leq N_1, 1\leq j\leq N_2:=|\det R_2|\right\},
\label{eq1.23.1}
\end{equation}
and $\{d_{i,j}:1\leq j\leq N_2\}$ ($d_{i,j} = v_i +Qc_{i,j}$ as in Theorem \ref{th_quasi}) is a complete set of representatives $(\mod R_2\bz^{d-r})$.
We will show that the measure $\mu=\mu(R,B)$ has a quasi-product structure.

\medskip

Note that we have
$$R^{-1}=\begin{bmatrix}
R_1^{-1}&0\\
-R_2^{-1}CR_1^{-1}&R_2^{-1}
\end{bmatrix}$$
and, by induction,
$$R^{-k}= \begin{bmatrix}
R_1^{-k}&0\\
D_k&R_2^{-k}
\end{bmatrix},\mbox{ where }D_k:=-\sum_{l=0}^{k-1}R_2^{-(l+1)}CR_1^{-(k-l)}.$$
 For the invariant set $T(R,B)$, we can express it as a set of infinite sums,
 $$
 T(R,B)=\left\{\sum_{k=1}^\infty R^{-k}b_k : b_k\in B\right\}.
 $$
Therefore any element $(x,y)^T\in T(R,B)$ can be written in the following form
$$x=\sum_{k=1}^\infty R_1^{-k}a_{i_k}, \quad y=\sum_{k=1}^\infty D_ka_{i_k}+\sum_{k=1}^\infty R_2^{-k}d_{i_k,j_k}.$$
Let $X_1$ be the attractor (in $\br^r)$ associated to the IFS defined by the pair $(R_1,\pi_1(B)=\{u_i :1\leq i\leq N_1\})$ (i.e. $X_1=T(R_1,\pi_1(B))$). Let $\mu_1$ be the (equal-weight) invariant measure associated to this pair.

\medskip

For each sequence $\omega=(i_1i_2\dots)\in\{1,\dots,N_1\}^{\bn} = \{1,\dots, N_1\}\times\{1,\dots, N_1\}\times...$, define
\begin{equation}\label{eqxomega}
x(\omega)=\sum_{k=1}^\infty R_1^{-k}u_{i_k}.
\end{equation}
 As $(R_1,\pi_1(B))$ forms Hadamard triple with some $L_1(\ell_2)$ by Lemma \ref{lem1.18}, the measure $\mu(R_1,\pi_1(B))$ has the no-overlap property (Theorem \ref{thDL15}). It implies that for $\mu_1$-a.e. $x\in X_1$, there is a unique $\omega$ such that $x(\omega)=x$. We define this as $\omega(x)$. This establishes a bijective correspondence, up to measure zero, between the set $\Omega_1:=\{1,\dots,N_1\}^{\bn}$ and $X_1$. The measure $\mu_1$ on $X_1$ is the pull-back of the product measure which assigns equal probabilities $\frac1{N_1}$ to each digit.

\medskip

For $\omega=(i_1i_2\dots)$ in $\Omega_1$, define
$$\Omega_2(\omega):=\{(d_{i_1,j_1}d_{i_2,j_2}\dots d_{i_n,j_n}\dots) : j_k\in \{1,\dots,N_2\}\mbox{ for all }k\in\bn\}.$$
For $\omega\in\Omega_1$, define $g(\omega):=\sum_{k=1}^\infty D_ka_{i_k}$ and $g(x):=g(\omega(x))$, for $x\in X_1$.  Also $\Omega_2(x):=\Omega_2(\omega(x))$.
For $x\in X_1$, define
$$X_2(x):=X_2(\omega(x)):=\left\{\sum_{k=1}^\infty R_2^{-k}d_{i_k,j_k}: j_k\in\{1,\dots,N_2\}\mbox{ for all }k\in\bn\right\}.$$
Note that the attractor $T(R,B)$ has the following form

$$T(R,B)=\{(x,g(x)+y)^T: x\in X_1,y\in X_2(x)\}.$$

\medskip

For $\omega\in \Omega_1$,  consider the product probability measure $\mu_\omega$, on $\Omega_2(\omega)$, which assigns equal probabilities $\frac{1}{N_2}$ to each digit $d_{i_k,j_k}$ at level $k$.
Next, we define the measure $\mu_\omega^2$ on $X_2(\omega)$. Let
$r_\omega:\Omega_2(\omega)\rightarrow X_2(\omega)$,
$$r_\omega(d_{i_1,j_1}d_{i_2,j_2}\dots)=\sum_{k=1}^\infty R_2^{-k}d_{i_k,j_k}.$$
Define $\mu_x^2:=\mu_{\omega(x)}^2:=\mu_{\omega(x)}\circ r_{\omega(x)}^{-1}$.

\medskip

Note that the measure $\mu_x^2$ is the infinite convolution product $\delta_{R_2^{-1}B_2(i_1)}\ast\delta_{R_2^{-2}B_2(i_2)}\ast\dots$, where $\omega(x)=(i_1i_2\dots)$, $B_2(i_k):=\{d_{i_k,j} : 1\leq j\leq N_2\}$ and $\delta_A:=\frac{1}{\#A}\sum_{a\in A}\delta_a$, for a subset $A$ of $\br^{d-r}$. The following lemmas were proved in \cite{DJ07d}.

\begin{lemma}\label{lem1.23}\cite[Lemma 4.4]{DJ07d}
For any bounded Borel functions on $\br^d$,
$$\int_{T(R,B)}f\,d\mu=\int_{X_1}\int_{X_2(x)}f(x,y+g(x))\,d\mu_x^2(y)\,d\mu_1(x).$$
\end{lemma}

\begin{lemma}\label{lem1.24}\cite[Lemma 4.5]{DJ07d}
If $\Lambda_1$ is a spectrum for the measure $\mu_1$, then
$$F(y):=\sum_{\lambda_1\in\Lambda_1}|\widehat\mu(x+\lambda_1,y)|^2=\int_{X_1}|\widehat\mu_s^2(y)|^2\,d\mu_1(s),\quad(x\in\br^r,y\in\br^{d-r}).$$
\end{lemma}

\medskip

We recall also the Jorgensen-Pedersen Lemma for checking in general when a countable set is a spectrum for a measure.

\begin{lemma}\label{lemJP}\cite{JP98}
Let $\mu$ be a compactly supported probability measure. Then $\Lambda$ is a spectrum for $L^2(\mu)$ if and only if
$$
\sum_{\lambda\in\Lambda}|\widehat{\mu}(\xi+\lambda)|^2\equiv 1.
$$
\end{lemma}

We need the following key proposition.

\begin{proposition}\label{lem1.25}
For the quasi-product form given in (\ref{R_4.1}) and (\ref{eq1.23.1}), there exists a lattice $\Gamma_2$ such that for $\mu_1$-almost every $x\in X_1$, the set $\Gamma_2$ is a spectrum for the measure $\mu_x^2$.
\end{proposition}

\begin{proof}
First we replace the first component $(R_1,\pi_1(B))$ by a more convenient pair which allows us to use the theory of self-affine tiles from \cite{LW2}. Define
$$R^\dagger:=\begin{bmatrix}
N_1&0\\
0&R_2
\end{bmatrix},$$
$$B^\dagger=\left\{(i,d_{i,j})^T: 1\leq i\leq N_1-1, 1\leq j\leq |\det R_2|\right\}.$$
We will use the super-script $\dagger$ to refer to the pair $(R^\dagger,B^\dagger)$.

\medskip

As $d_{i,j}$ is a complete residue ($\mod R_2(\bz^{d-r})$), the set $B^\dagger$ is a complete set of representatives $(\mod R^\dagger(\bz^{d-r+1}))$. By \cite[Theorem 1.1]{LW2}, $\mu^\dagger$ is the normalized Lebesgue measure on $T(R^\dagger,B^\dagger)$ and
this tiles $\br^{d-r+1}$ with some lattice $\Gamma^{\ast}\subset{\mathbb Z}^{d-r+1}$. The attractor $X_1^{\dagger}$ corresponds to the pair $(N_1,\{0,1,\dots,N_1-1\})$ so $X_1^\dagger$ is $[0,1]$ and $\mu_1^\dagger$ is the Lebesgue measure on $[0,1]$.  We need the following claim:

 \medskip

 {\it Claim:} the set $T(R^\dagger, B^\dagger)$ actually tiles $\br^{d-r+1}$ with a set of the form $\bz\times \tilde{\Gamma}_2$, where $\tilde\Gamma_2$ is a lattice in $\br^{d-r}$.

 \medskip

\noindent {\it Proof of claim:} This claim was established implicitly in the proof of Theorem 1.1, in section 7, p.101 of \cite{LW2}, we present it here for completeness. Let $\Gamma^{\ast}$ be the lattice on ${\mathbb R}^{d-r+1}$ which is a tiling set of $T(R^\dagger, B^\dagger)$. We observe that the orthogonal projection of  $T(R^\dagger, B^\dagger)$ to the first coordinate is $[0,1]$. Hence, for any $\gamma\in\Gamma^{\ast}$, the orthogonal projection of $T(R^\dagger, B^\dagger)$ is $[0,1]+\gamma_1$, where $\gamma= (\gamma_1,\gamma_2)^T$. As $\Gamma^{\ast}\subset{\mathbb Z}^{d-r+1}$. these projections $[0,1]+\gamma_1$ are measure disjoint for different $\gamma_1$'s. Therefore, the tiling of $T(R^\dagger, B^\dagger)$ by $\Gamma^{\ast}$ naturally divides up into cylinders:
$$
U(\gamma_1): = ([0,1]+\gamma_1)\times {\mathbb R}^{d-r}.
$$
Focusing on one of the cylinders, say $U(0)$, this cylinder is tiled by $\tilde{\Gamma}$ where
$$
\tilde{\Gamma} = \Gamma\cap (\{0\}\times {\mathbb Z}^{d-r}).
$$
As ${\mathbb R}^{d-r+1} = \bigcup_{\gamma_1}U(\gamma_1)$, this means $T(R^\dagger, B^\dagger)$  also  tiles by ${\mathbb Z}\times\tilde{\Gamma}_2$. This completes the proof of the claim.
\medskip

 Because of the claim, it follows from the well-known result of Fuglede \cite{Fug74} that $\mu^\dagger$ has a spectrum of the form $\bz\times\Gamma_2$, with $\Gamma_2$ the dual lattice of $\tilde\Gamma_2$.

\medskip

We prove that $\Gamma_2$ is an orthogonal set for the measure ${\mu^\dagger}_s^2$ for $\mu_1^\dagger$-almost every $s\in X_1^\dagger$. Indeed, for $\gamma_2\neq 0$  in $\Gamma_2$, since $\bz\times\Gamma_2$ is a spectrum for $\mu^\dagger$, we have for all $\lambda_1\in\bz$, with Lemma \ref{lem1.23},
$$0=\int_{T(R^\dagger,B^\dagger)}e^{-2\pi i \ip{(\lambda_1,\gamma_2)}{(x,y)}}\,d\mu^\dagger(x,y)=\int_{X_1^\dagger}\int_{{X^\dagger}_2(x)}e^{-2\pi i\ip{(\lambda_1,\gamma_2)}{(x,y)}}\,d{\mu^\dagger}_x^2(y)\,d\mu_1^\dagger(x)$$
$$=\int_{X_1^\dagger}e^{-2\pi i\lambda_1x}\int_{{X^\dagger}_2(x)}e^{-2\pi i\ip{\gamma_2}{ y}}\,d{\mu^\dagger}_x^2(y)\,d\mu^\dagger_1(x).$$
This implies that that
$$\int_{{X^\dagger}_2(x)}e^{-2\pi i \ip{\gamma_2}{y}}\,d{\mu^\dagger}_x^2(y)=0,$$
for all $\gamma_2\in\Gamma_2$ for $\mu^\dagger_1$-a.e. $x\in X_1^\dagger$. This means that $\Gamma_2$ is an orthogonal sequence for ${\mu^\dagger}_x^2$ for $\mu^\dagger_1$-a.e. $x\in X_1^\dagger$ so
\begin{equation}
\sum_{\gamma_2\in\Gamma_2}|\widehat{{\mu^\dagger}_x^2}(y+\gamma_2)|^2\leq 1,\quad(y\in\br^{d-r}),
\label{eq1.25.1}
\end{equation}
for $\mu^\dagger_1$-a.e. $x\in X_1^\dagger$.
With Lemma \ref{lem1.25}, we have
$$1=\sum_{\gamma_2\in\Gamma_2}\sum_{\lambda_1\in\bz}|\widehat\mu^\dagger(x+\lambda_1,y+\gamma_2)|^2=\int_{X^\dagger_1}\sum_{\gamma_2\in\Gamma_2}|\widehat{{\mu^\dagger}_s^2}(y+\gamma_2)|^2\,d\mu^\dagger_1(s).$$
With \eqref{eq1.25.1}, we have
\begin{equation}
\sum_{\gamma_2\in\Gamma_2}|\widehat{{\mu^\dagger}_s^2}(y+\gamma_2)|^2= 1,\quad(y\in\br^{d-r}),
\label{eq1.25.2}
\end{equation}
for $\mu_1^\dagger$-a.e. $s\in X_1^\dagger$,
which means that $\Gamma_2$ is a spectrum for almost every measure ${\mu^\dagger}_s^2$ by Lemma \ref{lemJP}.

\medskip

Now, we are switching back to our original pair $(R,B)$. Note that we have the maps $x:\Omega_1\rightarrow X_1$ and $x^\dagger:\Omega_1\rightarrow X_1^\dagger$, defined by $\omega\mapsto x(\omega)$ as above in \eqref{eqxomega}, and analogously for $x^\dagger$. The maps are measure preserving bijections. Let $\Psi:X_1\rightarrow X_1^\dagger$ be the composition $\psi=x^\dagger\circ x^{-1}$. i.e.
 $$
 \Psi\left(\sum_{j=1}^{\infty} R_1^{-j}u_j\right) = \sum_{j=1}^{\infty} N_1^{-j}j.
 $$
 Consider the measure $\nu (E) = \mu_1^{\dagger}(\Psi(E))$ for Borel set $E$ in $T(R_1,\pi_1(B))$. Because of the no-overlap condition, we can check easily that $\nu$ and $\mu_1$ agrees on all the cylinder sets of $T(R_1,\pi_1(B))$. i.e.
 $$
 \nu(\tau_{i_1}\circ...\circ\tau_{i_n}(T(R_1,\pi_1(B)))) = \frac{1}{N^n} =  \mu_1(\tau_{i_1}\circ...\circ\tau_{i_n}(T(R_1,\pi_1(B))))
 $$
 for all $i_1,...,i_n\in\{0,1,...,N_1-1\}$. This shows that $\nu = \mu_1$ and therefore $\mu_1(E) = \mu_1^{\dagger}(\Psi(E))$ for any Borel set $E$. Consider the set
 $$
 {\mathcal N} = \{x\in T(R_1,\pi_1(B)): \Gamma_2 \ \mbox{is not a spectrum} \ \mbox{for} \ \mu_x^2 \}
 $$
 Then
 $$
 \Psi({\mathcal N}) = \{\Psi(x)\in X_1^{\dagger}: \Gamma_2 \ \mbox{is not a spectrum} \ \mbox{for} \ \mu_x^2 \}
 $$
Note also that, on the second component, the two pairs $(R,B)$ and $(R^\dagger,B^\dagger)$ are the same, more precisely $X_2(x)=X_2^\dagger(\Psi(x))$ and $\mu_x^2={\mu^\dagger}_{\Psi(x)}^2$ for all $x\in X_1$. This means that
$$
 \Psi({\mathcal N}) = \{\Psi(x)\in X_1^{\dagger}: \Gamma_2 \ \mbox{is not a spectrum} \ \mbox{for} \ \mu_{\Psi(x)}^2 \}
$$
which has $\mu_1^{\dagger}$-measure 0, by the arguments in the previous paragraph. Hence, $\mu_1(E) = \mu_1^{\dagger}(\Psi(E))=0$ and this completes the proof.
\end{proof}

\medskip

\begin{proof}[Proof of Theorem \ref{thmain}]

To prove Theorem \ref{thmain},  we use induction on the dimension $d$. We know from \cite{LaWa02,DJ06} that the result is true in dimension one (See also \cite[Example 5.1]{DL15} for an independent proof by considering ${\mathcal Z}$). Assume it is true for any dimensions less than $d$.

\medskip

First, after some conjugation, we can assume that $\bz[R,B]=\bz^d$, according to Proposition \ref{przrb}. Next, if the set in \eqref{Z} $\mathcal Z=\emptyset$, then the result follows from Theorem \ref{thDL15} (ii). Suppose now that ${\mathcal Z}\neq\emptyset$. Then, by Proposition \ref{pr1.14}, we can conjugate with some matrix so that $(R,B)$ are of the quasi-product form given in (\ref{R_4.1}) and (\ref{eq1.23.1}).

\medskip

By Theorem \ref{th_quasi}, $(R_1,\pi_1(B), L_1(\ell_2))$ forms a Hadamard triple with some $L$ on ${\mathbb R}^{r}$ where $1\leq r<d$. By induction hypothesis, the measure $\mu_1$ is spectral.
Let $\Lambda_1$ be a spectrum for $\mu_1$. By Proposition \ref{lem1.25}, there exists $\Gamma_2$ such that $\Gamma_2$ is a spectrum for $\mu_x^2$ for $\mu_1$-almost everywhere $x$.  Then we have, with \eqref{eq1.25.2}, and Lemma \ref{lem1.24},
$$\sum_{\gamma_2\in\Gamma_2}\sum_{\lambda_1\in\Lambda_1}|\widehat\mu(x+\lambda_1,y+\gamma_2)|^2=\int_{X_1}\sum_{\gamma_2\in\Gamma_2}|\widehat{{\mu}_s^2}(y+\gamma_2)|^2\,d\mu_1(s)=\int_{X_1}1d\mu_1(s)=1.$$
This means that $\Lambda_1\times\Gamma_2$ is a spectrum for $\mu$ by Lemma \ref{lemJP} and this completes the whole proof of Theorem \ref{thmain}.
\end{proof}

\medskip

 \begin{acknowledgements}
This work was partially supported by a grant from the Simons Foundation (\#228539 to Dorin Dutkay).
\end{acknowledgements}

\bibliographystyle{alpha}	
\bibliography{eframes}

\end{document}